\documentclass[a4paper]{article}
\usepackage[english]{babel}
\usepackage{amsmath}
\usepackage{amssymb}
\usepackage{amsthm}
\usepackage{dsfont}
\usepackage{mathtools}
\usepackage{algorithm}
\usepackage{enumerate}
\usepackage[pdftex,colorlinks]{hyperref}
\usepackage{microtype}
\usepackage[]{xspace}
\hypersetup{%
	pdftex,%
	colorlinks,%
	pdfauthor = {Sep\ Thijssen},%
	pdftitle = {Consistent AMIS and controlled diffusions}%
}
\usepackage[bitstream-charter]{mathdesign}
\usepackage[T1]{fontenc}

\newcommand{\E}{\ensuremath{\mathbb{E}}}

\newcommand{\R}{\ensuremath{\mathbb{R}}}

\newcommand{\Var}{\operatorname{Var}}

\newcommand{\tint}{\ensuremath{\textstyle{\int}}}

\newtheorem{thrm}{Theorem}
\newtheorem{lemm}[thrm]{Lemma}
\theoremstyle{definition}
\newtheorem{deff}[thrm]{Definition}
\theoremstyle{remark}
\newtheorem{rmrk}[thrm]{Remark}
\newtheorem{exmp}[thrm]{Example}

\date{August 25, 2016}
\begin{document}

\title{\bf Consistent Adaptive Multiple Importance Sampling and Controlled Diffusions}
\author{
	Sep Thijssen
	\thanks{
		This work was funded by D-CIS Lab / Thales Research \& Technology NL, and supported by the European Community Seventh Framework Programme (FP7/2007-2013) under grant agreement 270327 (CompLACS). The authors would like to express their appreciation to Prof. dr. E.A. Cator for fruitful discussion.
		}
	\hspace{.2cm}\\
	Department of Neurophysics, Donders Institute for Neuroscience,\\
	Radboud University Nijmegen
	and \\
	H.J. Kappen \\
	Department of Neurophysics, Donders Institute for Neuroscience,\\
	Radboud University Nijmegen
}
\maketitle

\newpage

\begin{abstract}
Recent progress has been made with Adaptive Multiple Importance Sampling (AMIS) methods that show improvement in effective sample size. However, consistency for the AMIS estimator has only been established in very restricted cases. Furthermore, the high computational complexity of the re-weighting in AMIS (called balance heuristic) makes it expensive for applications involving diffusion processes. In this work we consider sequential and adaptive importance sampling that is particularly suitable for diffusion processes. We propose a new discarding-re-weighting scheme that is of lower computational complexity, and we prove that the resulting AMIS is consistent. Using numerical experiments, we demonstrate that discarding-re-weighting performs very similar to the balance heuristic, but at a fraction of the computational cost.
\end{abstract}

\noindent%

\section{Introduction}
\label{sect:intro}

Monte Carlo (MC) integration is a broadly applied method to numerically compute integrals that might be difficult to evaluate otherwise, due to, for example, high dimensions. The main shortcoming of MC integration is perhaps that the estimator can have a high variance, which has led to techniques such as Importance Sampling (IS). The idea behind IS is reducing the variance of an estimator by drawing samples from a chosen proposal distribution that puts more emphasis on ``important'' regions. This will in general introduce a bias, which has to be corrected with an importance weight. 

There are, generally speaking, two different motivations for implementing IS. 

First, one might be interested in an expected value over a distribution $Q$ from which it is impossible to draw samples (efficiently). In this case a proposal distribution can be constructed in order to generate samples \cite{cappe, marin, cornuet}. When the density $q = dQ/dx$ is only known up to a factor, the normalization constant needs to be estimated as well. For this reason, it is common to choose a proposal distribution close to $Q$.

The second motivation to use IS, is whenever sampling from $Q$ is possible but very inefficient for the purpose of MC integration. This is, for example, typically the case with conditioned diffusions \cite{doob} or stochastic control problems \cite{ruiz}, which have important applications of IS in, for example, robotics, \cite{theodorou}.
Our motivation to use IS is of the second kind.

In cases where it is difficult to choose a single proposal distribution that covers all the important regions, one can resort to a mixture of proposal distributions, This technique is known as Multiple Importance Sampling (MIS) \cite{owen}. An important problem in MIS is the choice or construction of good proposal distributions. Roughly, there are two approaches: either the proposals are carefully chosen in advance of the sampling procedure \cite{veach}, or the proposals are optimized during the sampling procedure \cite{oh, cornuet}. The advantage of the former is that it is clearly consistent, because, in contrast to the latter, all samples are independent. The advantage of the latter is that the optimization scheme might yield better sampling efficiency. 

A particular instance of MIS with optimization of the proposals during sampling is the so-called Adaptive Multiple Importance Sampling (AMIS) algorithm \cite{cornuet}. In AMIS the samples and their associated importance weights are combined according to the balance heuristic. Although the balance heuristic is optimal in the sense of variance reduction when the number of samples goes to infinity \cite{veach}, it also introduces a complicated dependence between the samples from the various proposals. As a consequence, consistency for AMIS is a non trivial proposition, which only recently been established, and only in restricted cases \cite{marin}. 

An aspect of AMIS, or more generally of MIS, that has not been addressed by the literature, is that of the additional computational overhead that is caused by re-weighting of the samples. This overhead is proportional to the cost of computing a likelihood ratio. In some scenarios, for example when sampling requires real world interaction, this cost might be negligible. However, in MC sampling this cost will be roughly proportional to the cost of drawing a sample. 
This becomes an issue when the re-weighting scheme has a higher computational complexity than the drawing process, because in that case the algorithm will eventually spend more time on re-weighting than on drawing samples. Critically, this is the case when re-weighting uses the balance heuristic, which has a complexity of $\mathcal {O}(K^2M)$, where $K$ is the number proposal distributions, and $M$ the number of samples per proposal. Note that this is larger than the complexity $\mathcal {O}(KM)$ of drawing all the $MK$ samples, particularly with many proposal distributions. 

In this work we propose a new re-weighting scheme, called discarding-re-weighting, and address the issues described above. 
In particular, discarding-re-weighting will have a complexity of $\mathcal {O}(KM)$. We will provide a consistency proof of the corresponding discarding-AMIS, without any restrictions, aside from the usual, on the proposals distributions. 
Furthermore, we will show in a numerical example that our proposed re-weighting scheme is well suited for sampling over diffusion processes, by comparing it with the balance heuristic.

\paragraph{Outline.} 
The remainder of this work is structured as follows. In Section~\ref{sect:gen-amis} we review the generic AMIS method. Sections~\ref{sect:flat}--\ref{sect:wiener} consider the re-weighting scheme, where in Section~\ref{sect:flat} we treat consistency, in Section~\ref{sect:discard} introduce discarding-re-weighting, which we apply in Section~\ref{sect:wiener} to sampling over diffusion processes. In Section~\ref{sect:control} we propose a specific proposal update in the context of diffusion processes. This update is used in Section~\ref{sect:exmp} to compare our new re-weighting scheme with the balance heuristic.

\section{The generic AMIS}\label{sect:gen-amis}

In this section we briefly review IS, MIS and AMIS for MC integration. In particular we shall give a description of a generic AMIS. 

Let $(\Omega, \mathcal{F}, Q)$ be a probability space with an $E$-valued random variable $X$, and an $\mathbb{R}$-valued function $h(X)$. The goal is to calculate 
\begin{align*}
\psi = \E_Q[h(X)],
\end{align*}
using a MC estimate.
In particular we will be interested in variance reduction that can be achieved via importance sampling. Let $P$ be another probability measure on $(\Omega, \mathcal{F})$, and let $dQ/dP$ denote a density of $Q$ relative to $P$, then
\begin{align}
\hat\psi^P = \frac1N\sum_{n = 1}^N h(X_n)\tfrac{dQ}{dP}(X_n),\qquad \text{ where } X_n\sim P, \label{eq:is}
\end{align}
is an unbiased estimator for $\psi$, provided that for all events $A\in \mathcal{F}$
\begin{align}
P(A) = 0 \quad \Longrightarrow \quad h = 0 \text{ on } A, Q\text{-almost surely.}\label{eq:p<<qh}
\end{align}
Often condition~(\ref{eq:p<<qh}) is replaced by the stronger assumption of absolute continuity, $Q\ll P$, so that the importance weight $dQ/dP$ exists everywhere. Regarding importance sampling, however, we only require that $dQ/dP$ exists whenever $h \neq 0$.

Instead of using one proposal, $P$, the MC estimate can also be based on a mixture of proposals. For $k = 1, \ldots, K$ let $P_k$ be probability measures on $(\Omega, \mathcal{F})$ all satisfying Condition~(\ref{eq:p<<qh}). The Multiple IS (MIS) estimator is defined as
\begin{align}\label{eq:mis}
\hat\psi^{\operatorname{MIS}} = \frac1N\sum_{k = 1}^{K} \sum_{n = 1}^{N_k} h(X_n^k) \tfrac{dQ}{dP_k}(X_n^k) w_k(X_n^k),	\\
X_n^k\sim P_k,\qquad\textrm{for }n = 1, \ldots, N_k,\nonumber
\end{align}
where $N = \sum_{k = 1}^K N_k$ is the total number of samples. 
If the $X_n^k$ are independent, and the re-weighting functions $w_k(x)$ satisfy
\begin{align*}
h(x) \neq 0 &\quad\Rightarrow\quad \frac1N \sum_{k = 1}^K N_k w_k(x) = 1,
\end{align*}
then $\hat\psi^{\operatorname{MIS}}$ is an unbiased estimate \cite{veach}. Remarkably there are many choices for $w_k$. A particularly simple choice would be $w_k = 1$, which will henceforth be referred to as flat re-weighting. Another scheme that is of interest is the so called balance-heuristic, which is also called deterministic multiple mixture. It is defined by
\begin{align*}
w_k(x) = \frac{1}{\frac1N \sum_{l = 1}^K N_k \tfrac{dP_l}{dP_k}(x)}.
\end{align*}
The advantage of balance heuristic over flat re-weighting is that the former results in lower variance mixed estimates when combining, for example, a (good) proposal that gives low variance estimates with a (bad) proposal that gives high variance estimates. The reason is, roughly speaking, that the reciprocal of the variance of the balance mix is the reciprocal of the harmonic mean of the variance of the individual proposals, while for the flat re-weighted mix this is the standard arithmetic mean. For a study on the relative merits of various related re-weighting schemes for MIS see \cite{veach, owen}. 

In order to improve the efficiency of a MIS algorithm, one can adapt the proposals sequentially. This idea was first mentioned in \cite{oh} with the name Adaptive Importance Sampling (AIS), and more recently in \cite{cornuet} with the name Adaptive Multiple Importance Sampling (AMIS). Both of these methods adapt the proposals at iteration $k$ by adapting a parameter that is estimated using all samples that are draw up to iteration $k$. The two methods differ in the re-weighting: AMIS uses the balance-heuristic, while AIS uses flat re-weighting. If we instead consider the idea of adaptive sequential updates without specifying the form of the proposal or the re-weighting scheme we obtain a generic AMIS \cite{elvira, martino}:
\begin{algorithm}
\caption{generic AMIS}\label{algo:gamis}
\begin{itemize}
\item
At iteration $k = 1, \ldots, K $ do
	\begin{description}
	\item[Adaptation.] 
		Construct a measure $P_{k}$, possibly depending on $X_n^l$ and $w_n^l$ with $1\leq l < k$, $1\leq n\leq N_l$\label{algo:amisupdate}
	\item[Generation.] 
		For $n = 1, \ldots, N_k$ draw $X_n^k\sim P_{k}$
	\item[Re-weighting.] 
		For $n = 1, \ldots, N_k$ construct $w_n^k$.\\
		For $l = 1, \ldots, k- 1$ and $n = 1, \ldots, N_l$, update $w_n^l$.
	\item[Output.] Return $\hat\psi_k = \frac1{\sum_{l= 1}^kN_l}\sum_{l = 1}^k\sum_{n = 1}^{N_l} \tfrac{dQ}{dP_l}(X_n^l)h(X_n^l)w_n^l$
	\end{description}
\end{itemize}
\end{algorithm}

\index{AMIS} 
The computational complexity of the generic AMIS will depend on the specifics of both the adaptation and re-weighting step. For example, AMIS with $K$ iterations that uses the balance-heuristic has a complexity of $\mathcal{O}(MK^2)$, when $N_k = M$ samples are used at each iteration $k$, while for flat re-weighting this is only $\mathcal{O}(MK)$.

The unbiasedness and consistency from MIS does in general not carry over to the generic AMIS. The adaptation step introduces dependencies between samples from different iterations. Furthermore, the re-weighting might introduce extra correlations. Consistency has been established for a specific AMIS in \cite{marin} under the assumption that the adaptation is only based on the last $N_k$ samples and that $N_k$ grows at least as fast as $k$. The downside of this method is that the proposal cannot be updated very frequently, and only while using a subset of all the available samples. In the next section we will establish consistency of AMIS with flat re-weighting (flat-AMIS) for generic proposal adaptations without any such restrictions. 

\section{Consistency of flat-AMIS}\label{sect:flat}

In this section we will prove that flat-AMIS is consistent. Consistency can only be established when we make some assumptions on the proposals (see Example~\ref{exmp:counter}), but these assumptions will be quite general and they often do not pose any restrictions in practice.

Let $\mathcal{P}$ be the class of proposal distributions. Let $\|X\|_r = \left(\E\left[|X|^r\right]\right)^{1/r}$ denote the $L^r$-norm. We will require that there are constants $r>1$ and $C>0$, such that for all $P\in\mathcal{P}$
\begin{align} \label{eq:bound}
\left\|h(X)\tfrac{dQ}{dP}\right\|_r < C, \qquad X\sim P.
\end{align}
\begin{thrm}[Flat-AMIS is consistent] \label{thrm:consistent}
Let $\hat\psi_k$ be defined as in the output step of Algorithm~\ref{algo:gamis} using flat re-weighting, i.e.~with $w_n^k = 1$. Suppose that both Eq.~(\ref{eq:p<<qh}) and (\ref{eq:bound}) are satisfied, then
\begin{align}
\hat\psi_k\to \psi\qquad \text{a.s.} 
\end{align}
when $\sum_{l\leq k} N_l\to\infty$. 
\end{thrm}
\begin{proof}
Let $i(n, k) = n + \sum_{l<k}N_l$ denote the total number of samples so far, and define $Y_{i} = h(X_n^k)\tfrac{dQ}{dP_k}(X_n^k)$. Then by Eq.~(\ref{eq:p<<qh}) we obtain that $Y_i$ is an unbiased estimator of $\psi$, when conditioning on all samples up to $i$, i.e.~$\E[Y_i\mid X_j, j<i] = \psi$. Therefore $\{Y_i - \psi\}_{i>0}$ is a martingale difference sequence (see Definition~\ref{deff:mds}). Furthermore, by the Minkowski inequality, we get that $\|Y_i - \psi\|_r \leq \|Y_i \|_r + \|\psi\|_r < C + \psi$, where we used Eq.~(\ref{eq:bound}) for the last inequality. We conclude that $Y_i - \psi$ is bounded uniformly in the $L^r$-norm. By Theorem~\ref{thrm:slln}, we obtain $I^{-1}\sum_{i = 1}^I Y_i\to \psi$ almost surely as $I\to\infty$. Now note that $\hat\psi_k = I^{-1}\sum_{i=1}^I Y_i$ when $I = I(N_k, k) = \sum_{l\leq k}N_l$. 
\end{proof}

Note that in the proof above we did not make any assumptions about the relative size between $k$ and $N_k$. In particular the result is valid in the two extreme cases when $N_k = 1$ for all $k$ and $K\to\infty$, or when $K$ is finite and $N_k\to\infty$ for any $k$.

\begin{exmp}\label{exmp:counter} 
Here we show that the condition of Eq.~(\ref{eq:bound}) in Theorem~\ref{thrm:consistent} is not redundant, by giving a sequence of proposals that will not yield a consistent estimate. 
Specifically, we consider the sampling problem that is given by
\begin{align*}
q(x) 
	&= \frac{dQ}{dx} =\frac1{\sqrt{2\pi}}\exp\left(-\tfrac12 x^2\right)\\
h(x) 
	&= \exp\left(-\tfrac12 x^2\right).
\end{align*}
We will consider the class $\mathcal{P} = \{P^u\colon u\in\mathbb{R}\}$ of proposal distributions, where
\begin{align*}
p^u(x) 
	&= \frac{dP^u}{dx} = \frac1{\sqrt{2\pi}}\exp\left(-\tfrac12 (x - u)^2\right).
\end{align*}
In Figure~\ref{fig:notui} we give a graphical representation of the importance sampling situation, with parameters $u = 1, 2, 3, 7$. Here you can see that the value of $h(x)q(x)/p_u(x)$ gets smaller in regions where $p^u(x)$ is large, when $u$ increases (compare the dashed and the solid line). 
Indeed, it is not difficult to prove that for all $\gamma>0$ we have $\lim_{u\to\infty}P^u(|h(X)\tfrac{dQ}{dP^u}(X)|>\gamma) = 0$. 
So if we take $u_k = k$, and $N_k =1$ and consider the flat-AMIS estimate $\hat\psi_K = \sum_{k = 1}^K h(X^k)\tfrac{dQ}{dP^k}(X^k)$, where $X^k\sim P^k = P^{u_k}$, then also
$\lim_{K\to\infty}\operatorname{Pr}(\hat\psi_K>\gamma) = 0$ for all $\gamma$. In words: the estimator $\hat\psi_K$ goes to zero in probability when $K\to\infty$. 
In contrast, for all $u$, we have $\psi = \E_{P^u}[h(X)\tfrac{dQ}{dP^u}] = 1/\sqrt{2}$ (area under the dotted line in Figure~\ref{fig:notui}). 
So, per definition, $\hat\psi_K$ is not a consistent estimator of $\psi$. 

\end{exmp}

\begin{figure}
\begin{center}
\includegraphics[width=10cm]{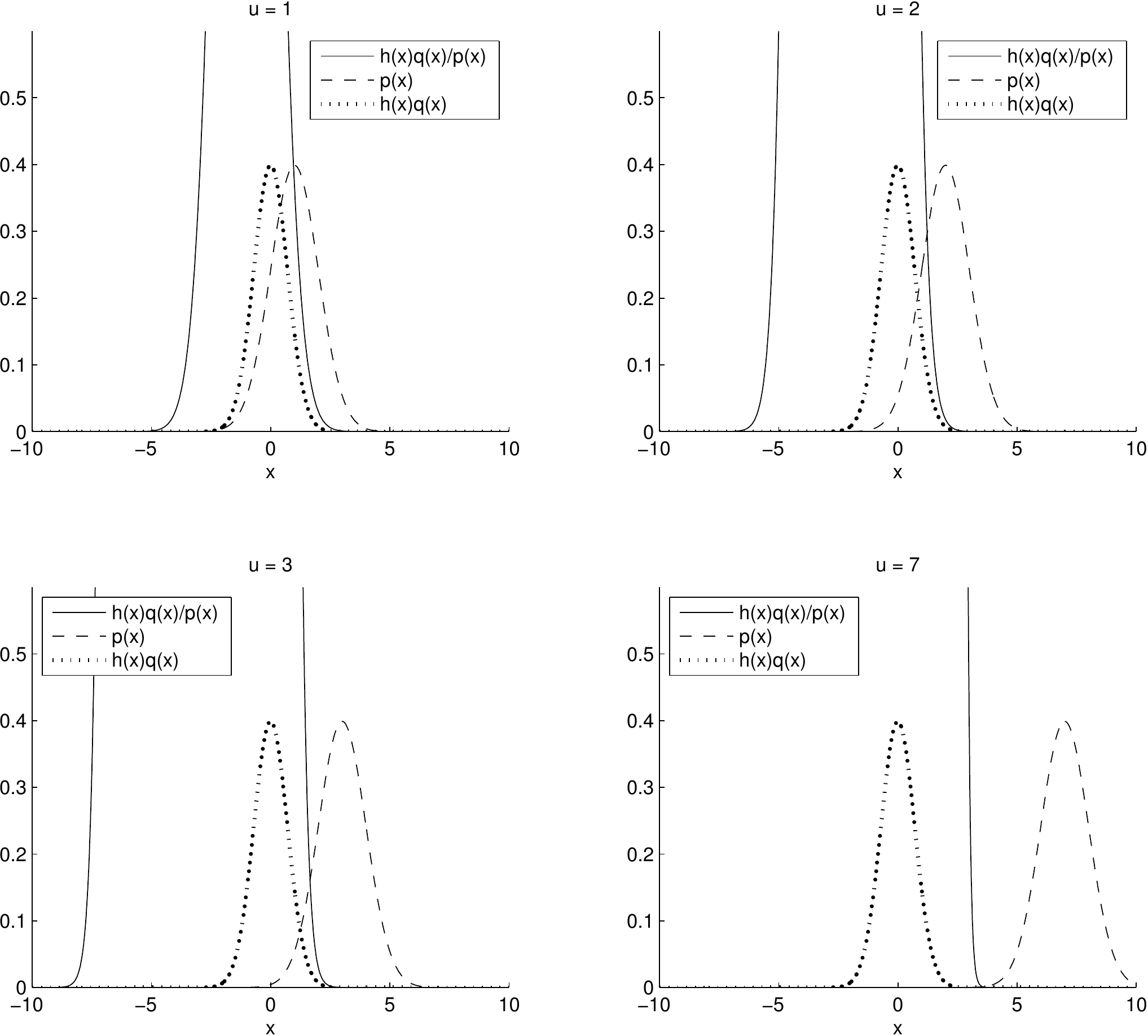}
\end{center}
\caption{
	For $u = 1, 2, 3, 7$ we plot $h(x)q(x)/p^u(x)$ (solid line), $p^u(x)$ (dashed line) and the product $h(x)q(x)$ (dotted line). Although the overlap between $h(x)q(x)/p^u(x)$ and $p^u(x)$ becomes smaller for larger $u$, the product does not depend on $u$.
}
\label{fig:notui}
\end{figure}

Theorem~\ref{thrm:consistent} holds for the generic proposal adaptation step in Algorithm~\ref{algo:gamis}, and condition Eq.~(\ref{eq:bound}) is the weakest that we were able to find. As a consequence, Eq.~(\ref{eq:bound}) is rather abstract and it might be hard to verify in practice. Therefore it might be sensible to replace Eq.~(\ref{eq:bound}) with a stronger condition that is easier to verify. We will do this for diffusion processes in Section~\ref{sect:wiener}.

\section{AMIS with discarding}\label{sect:discard}

Flat-AMIS yields, in contrast to balance-AMIS, a provably consistent estimate, see Theorem~\ref{thrm:consistent}. Furthermore, the computational complexity of flat-AMIS, is $\mathcal{O}(MK)$, when $N_k = M$ for all $k$, which is optimal, while the complexity of balance-AMIS is $\mathcal{O}(MK^2)$. Nevertheless, balance-AMIS will outperform flat-AMIS in most practical applications. The reason is that in a flat re-weighting scheme, samples from a poor proposal typically dominate the computation, while this effect is averaged out by the balance-heuristic. In this section we will show that a simple modification of the flat re-weighting scheme results in an AMIS that is both consistent and computationally efficient. 

The issue with flat re-weighting can be understood in more detail as follows. For a good proposal $P_1$ the terms $h\frac{dQ}{dP_1}$ do not deviate much from $\psi$. For a bad proposal $P_2$ most terms $h\frac{dQ}{dP_2}$ are close to zero, while a few will be exceptionally large compared to $\psi$. These large terms obviously dominate the IS estimate with $P_2$, but when mixing $P_1$ and $P_2$, the large terms from $P_2$ will also dominate over the samples from $P_1$. As a result, the mixing estimate might be worse than the IS estimate from the $P_1$ samples alone. 

To improve upon flat re-weighting we therefore propose to simply ignore the samples from bad proposals. Since the idea of AMIS is that with each adaptation the proposal improves, one will expect that the variance decreases over time, and the quality of the samples improves. This brings us to the following algorithm which we will call discarding-AMIS, where we specifically choose the following re-weighting step.  

\begin{description}
\item[Discarding-re-weighting] (at iteration $k$)\\
	Determine a discarding time $t_k\in\{1, 2, \ldots, k - 1\}$. \\
	For $l = 1, \ldots, t_k$ and $n = 1, \ldots, N_l$, set $w_n^l = 0$.\\
	For $l = t_k + 1, \ldots, k$ and $n = 1, \ldots, N_l$, set $w_n^l = k/(k - t_k)$.
\end{description}
Note that with this re-weighting the output at iteration $k$ of discarding-AMIS is
\begin{align*}
\hat\psi_k = \frac1{\sum_{l = t_k + 1}^kN_l}\sum_{l = t_k + 1}^k\sum_{n = 1}^{N_l} h(X_n^l) \tfrac{dQ}{dP_l}(X_n^l).
\end{align*}

The discarding time $t_k$ as given above is generic. We will now discuss two specific implementations of $t_k$ that both have their merits.

The first choice is motivated by the consistency issue. 
\begin{rmrk}
Theorem~\ref{thrm:consistent} still holds whenever $t_k$ is chosen independently of the sampling process and when $\sum_{l = t_K + 1}^KN_l\to\infty$. For example, one can take $t_k = \lceil k/2\rceil$ so long as $k\to\infty$.
\end{rmrk}

Secondly, let us consider a discarding time that aims to re-cycle the samples as efficiently as possible. When we have a measure of performance, we can utilize it to dynamically choose a discarding time that leaves us with the samples that yield the highest performance. For example, at iteration $k$ we can calculate the Effective Sample Size (ESS, see Eq.~(\ref{eq:sample_ess})) for all possible discarding times, and then choose the one that maximizes ESS. Clearly this will introduce a new level of dependence so that Theorem~\ref{thrm:consistent} no longer holds, and consistency is not guaranteed.
The computational cost of checking the ESS for all discarding times at iteration $k$ is $\mathcal{O}(Mk)$. If we do this at each iteration $k = 1, \ldots, K$, we get a total complexity of $\mathcal{O}(MK^2)$, which is more than $\mathcal{O}(MK)$ for the computations of the weights of all samples over all iterations. The latter however, might have a much larger prefactor, so that in practice the cost for finding the best ESS is negligible. This is for example the case with diffusion processes. Alternatively, one could consider the ESS for a sparser set of possible discarding times, such as $t = 2^s$ for $s = 1, 2, \ldots, \log(K)$, which will yield a complexity of $\mathcal{O}(MK\log(K))$. 

In Section~\ref{sect:exmp} we illustrate the difference in efficiency between $t_k = \lceil k/2\rceil$ and ESS-optimized discarding.

\section{Consistent AMIS for diffusion processes}\label{sect:wiener}

In this section we apply AMIS in order to compute expected values over a diffusion process, i.e.~with respect to the Wiener measure. By adding a drift to a diffusion process, we obtain a change in measure, and hence proposals that can be used for AMIS. We will give an easy to verify condition, involving the drift, that ensures consistency of flat-AMIS. 

In case of Wiener noise, the target measure $Q$, will implicitly be given by an $d$-dimensional It\^o process of the form
\begin{align}\label{eq:sdeq}
dX_t = \mu_tdt + \sigma_tdW_t, 
\end{align}
with $(\mu_t)_{0\leq t \leq T}$ and $(\sigma_t)_{0\leq t \leq T}$ adapted processes of dimension $d$ and $d\times m$ respectively, and $W_t$ an $m$-dimensional Brownian motion. The function $h$ in $\E_Q[h(X)]$ can be any function of the entire path: $h(X) = h\left((X_t)_{0\leq t\leq T}\right)$. 

If we have an adapted $m$-dimensional process $(u_t)_{0\leq t \leq T}$, we can implement IS with the proposal $P^u$ that is implicitly given by
\begin{align}\label{eq:sdep}
dX_t = \mu_tdt + \sigma_t\left(u_tdt + dW_t\right).
\end{align}
Often the adapted processes are given as feedback functions: $u_t = u(t, X_t)$, $\mu_t=\mu(t, X_t)$, $\sigma_t=\sigma(t, X_t)$. Instead of an explicit formula for the densities densities $dP^u/dx$, $dQ/dx$ with respect to a reference (e.g.~Lebesgue) measure $dx$, we only have access to stochastic differential equations such as Eq.~(\ref{eq:sdep}). On the upside, we will be able to generate (approximate) samples, for example by using the Euler-Maruyama method, \cite{kloeden}. So the goal in this scenario is not to generate samples close to the target $Q$; we can already do that by choosing $u = 0$. Instead, the aim of IS, or more generally, of AMIS, in this context, is to reduce the variance in the MC estimate of $\E_Q[h(X)] = \E_{P^u}[h(X)dQ/dP^u]$. In case of Wiener noise we are able to compute the importance weight $dQ/dP^u$, which, by the Girsanov Theorem \cite{karatzas, oksendal}, is given by:
\begin{align}\label{eq:rn}
\frac{dQ}{dP^u}
= \exp\left(-\int_0^T u_t^\top dW_t - \frac12\int_0^T u_t^\top u_tdt\right), 
\end{align}
where we have used $\top$ to denote the transpose. Note that since this equation is exact, we do not have to worry about normalization. 
\\ 

Next, we will investigate consistency of flat-AMIS in case of Wiener noise. Let $\mathcal{U}$ be a class of of adapted processes $(u_t)_{0\leq t\leq T}$ and let $\mathcal{P} = \{P^u\mid u\in \mathcal{U}\}$ be the corresponding class of proposal measures. We will replace the abstract conditions Eq.~(\ref{eq:p<<qh},~\ref{eq:bound}), that appear in Theorem~\ref{thrm:consistent}, by some assumptions that, although stronger, are easier to verify in practice. 
\begin{thrm}\label{thrm:consistent_u} Let $\mathcal{U}$ be a class of adapted processes. Suppose that 
\begin{enumerate}
\item $\mathcal{U}$ is uniformly bounded in the $L^\infty$ norm.\label{cond:ubound}
\item There is an $r>1$ such that $h\in L^r(Q)$.\label{cond:hbound}
\end{enumerate}
Then flat-AMIS with proposals from the class $\mathcal{P} = \{P^u\mid u\in \mathcal{U}\}$ is consistent. 
\end{thrm}
\begin{proof}
See Appendix~\ref{appendix:proof}. 
\end{proof}

If the adapted processes $u\in\mathcal{U}$ are given by feedback functions $u_t = u(t, X_t)$, then Condition~\ref{cond:ubound} is, for example, satisfied when 
$\mathcal{U}$ is uniformly bounded. 
Similarly, if $h$ is of the form $h(X) = \int_0^TH(X_t)dt$, or of the form $h(X) = H(X_T)$, for some function $H$, then Condition~\ref{cond:hbound} is satisfied if 
$H$ is bounded.

\section{Path integral adaptation}\label{sect:control}
 
In this section, we propose a specific adaptation step for Algorithm~\ref{algo:gamis} that can be used to sample over diffusion processes. We will adapt the proposal $P^u$ by estimating a `good' feedback function $u(t, x)$ that we can use in Eq.~(\ref{eq:sdep}). Here we interpret `good' as a function $u$ such that $P^u$ is close to an optimal proposal $P^{\star}$. For this optimal proposal to exist, we will assume for the remainder of this section that the function $h$ is strictly positive. Note that if this is not the case, one can consider $h = (h_+  + 1) - (h_- + 1)$, where $h_+(x) = \max(h(x), 0)$ and $h_-(x) = \max(-h(x), 0)$, and compute $\E_Q[h_+ + 1]$ and $\E_Q[h_- + 1]$ separately.

Since $h$ is strictly positive and $\E_Q[h(X)]<\infty$, the equation 
\begin{align}\label{eq:p*}
\frac{dP^{\star}}{dQ} = \frac{h(X)}{\E_Q[h(X)]},
\end{align}
defines a measure $P^{\star}$ that is equivalent to $Q$, which means, loosely speaking, that their densities have the same support. The measure $P^{\star}$ is the optimal proposal because IS with $P^{\star}$ gives zero variance estimates $h(X)dQ/dP^{\star} = \E_Q[h(X)]$. Note that, for the time being, this optimality is not of practical interest, since the definition of $P^{\star}$ requires $\E_Q[h(X)]$, which is what we want to evaluate in the first place. 

One might wonder whether there exists an optimal process $(u^{\star}_t)_{0\leq t\leq T}$ satisfying $P^{\star} = P^{u^{\star}}$.
This problem is actually studied in stochastic optimal control theory. In our particular scenario the answer is yes, this $u^\star$ exists, and it can be expressed in terms of a feedback function of the path $u^\star_t = u(t, X)$, provided that $\E_Q[h(X)|\log(h(X))|]<\infty$, see Proposition 3.5(ii) in \cite{bierkens}. 
This inspires us to use optimal control computations for the proposal update in the \textbf{Adaptation} step in the generic AMIS as given in Algorithm~\ref{algo:gamis}. 
Below we present a theorem from path integral control theory \cite{thijssen} that can be applied to estimate a `good' proposal-feedback-function $u(t, x)$ and so yields a path integral adaptation. We will demonstrate the use of the path integral adaptation by implementing various AMIS algorithms for an example sampling problem over a diffusion process.

\begin{thrm}\label{thrm:control}
Let $u_t$ be a control process, and let $B_t$ be a $P^u$-Brownian motion. Suppose that an optimal control $u^{\star}$ exists. Let $g_t$ be a $k$-dimensional measurable, square integrable and adapted process, then 
\begin{align}
\E_{P^u}\left[h(X)\frac{dQ}{dP^u} \int_0^T g_t {u^{\star}_t}^\top dt\right]
= \E_{P^u}\left[h(X)\frac{dQ}{dP^u} \int_0^T g_t \left(dB_t + u_tdt\right)^\top\right]. \label{eq:u*is}
\end{align}
\end{thrm}
\begin{proof}
See Appendix~\ref{appendix:proof}.
\end{proof}

Next, we illustrate how Theorem~\ref{thrm:control} can be used to construct a feedback controller $\hat u$, that will in a sense approximate $u^{\star}$. The starting point is that we choose a linear parametrization of the form 
\begin{align*}
\hat u(t, x) = Ag(t, x).
\end{align*}
Here $g(t, x)\in\mathbb{R}^l$ is an $l$-dimensional basis function that should be chosen by the user in advance of running the path integral control algorithm, and $A\in\mathbb{R}^{m\times l}$ is a parameter that shall be optimized by the algorithm. The idea behind the algorithm is to approximate $u^{\star}$ by $\hat u$ by optimizing over $A$. We can do this, for given $g$ (and a given importance sampling control $u$), if we substitute $\hat u$ for $u^{\star}$ in Eq.~(\ref{eq:u*is}). This will yield a system of equations that can be solved for $A$ as follows.
\begin{align}\label{eq:a}
A^{\star} = \E_{P^u}\left[h(X)\frac{dQ}{dP^u} \int_0^T (u_tdt + dW_t)g_t^\top\right] \E_{P^u}\left[h(X)\frac{dQ}{dP^u} \int_0^T g_tg_t^\top dt\right]^{-1}. 
\end{align}

The solution $A^{\star}$ is optimal in the sense that the corresponding $P^{\hat u^{\star}}$ (where $\hat u^{\star}=A^{\star}g$) minimizes the Kullback-Leibler divergence between $P^{\star}$ and $P^{\hat u}$, i.e. $\operatorname{D_{KL}}(P^{\star}\| P^{\hat u})$, over the class of proposal feedback functions with the given parametrization $\{\hat u = Ag\mid A\in \R^{m\times l}\}$, see \cite{ruiz, boer}. 

The smallest possible divergence $\operatorname{D_{KL}}(P^{\star}\| P^{\hat u})$ will depend on the function $g(t, x)\in\mathbb{R}^l$. Generally, complex $g$, i.e.~with $l$ large, yield more expressive power. In practice, however, there is a trade-off, since it is difficult to obtain good estimates of Eq.~(\ref{eq:a}), when $l$ is large. From a more practical point of view, it should be noted that the scenario where the algorithm is applied might put constraints on $g$. Whether or not that is the case, it is clear that the choice of the function $g$ is very important, because it determines what kind of controller you will create. For example, two types of controllers, which have perhaps been applied the most, are (1) the open-loop feed forward controller, and (2) a linear feedback controller. The (time constant) open-loop controller can be realized with $g = 1$, and the linear feedback controller with $g = (1, x)$. Note that time dependence can be introduced by using piecewise time-constant controls, i.e.~with functions of the form $g(t, x) = \sum_i g(x)\mathbb{1}_{t \in \Delta_i}$, where the $\Delta_i$ are small time intervals that cover $[t_0, t_1]$.

Eq.~(\ref{eq:a}) can be used for the proposal update in the \textbf{Adaptation} step in the generic AMIS as given in Algorithm~\ref{algo:gamis}. In Section~\ref{sect:exmp} we use this adaptation step in order to implement various AMIS algorithms for diffusion processes. A detailed description of this adaptation step is given in Appendix~\ref{appendix:algo}.

\section{Numerical example}\label{sect:exmp}

In this Section we provide a numerical example in which we compare discarding-AMIS with balance-AMIS. In both cases the adaptation step will be implemented as proposed in Section~\ref{sect:control}. 

We compare the various re-weighting schemes of AMIS in terms of Effective Sample Size (ESS). In the literature \cite{cornuet, 1511.03095} this is often defined by $N/\left(1 + \Var_P\left[\frac{dQ}{dP}\right]\right)$. However, since our goal is to minimize the variance of $h\frac{dQ}{dP}$, we instead consider 
\begin{align*}
\operatorname{ESS}^{P} = \frac{N}{1 + \Var_P\left[h\tfrac{dQ}{dP}\right]\psi^{-2}} 
= \frac{N}{1 + N\Var_P\left[\hat\psi^{P}\right]\psi^{-2}},
\end{align*}
where the second equality follows from $\Var_P[h\frac{dQ}{dP}] = N\Var_P[\hat\psi^P] $,which is a consequence of the definition in Eq.~(\ref{eq:is}).
We remark that this ESS is used in \cite{thijssen} in a setting with diffusion processes, and there it is shown that proposals $P^u$ with $u$ close to $u^\star$ have a close to optimal ESS, and vise versa. 
Similar to \cite{1511.03095}
we generalize the ESS for MIS:
\begin{align*}
\operatorname{ESS}^{\operatorname{MIS}} = \frac{N}{1 + N\Var_P\left[\hat\psi^{\operatorname{MIS}}\right]\psi^{-2}},
\end{align*}
which can be evaluated approximately with the following sample estimate
\begin{align}\label{eq:sample_ess}
\widehat{\operatorname{ESS}} 
	&= \frac{\left(\sum_{nk} y_{nk}\right)^2}{\sum_{nk} \left(y_{nk}\right)^2},
\end{align}
where $y_{nk} = h(X_n^k)\tfrac{dQ}{dP_k}(X_n^k)w_k(X_n^k)$, with $X_n^k\sim P_k$. The estimator $\widehat{\operatorname{ESS}}$ takes values between $1$ (when all but one $y_{nk}$ are zero) and $N$ (when all $y_{nk}$ are equal, which happens with positive probability iff $P = P^{\star}$). 

In the following example we will describe a sampling problem that will be used to compare discarding-AMIS with balance-AMIS.

\begin{exmp}\label{exmp:bench}
This example is similar to Example~\ref{exmp:counter}, where we interpret $X$ as a diffusion process, and generalize to $\R^d$. The target measure $Q$ is implicitly given by an $d$-dimensional standard Brownian motion. This is the process $(X_t)_{0\leq t\leq 1}$ as given in Eq.~(\ref{eq:sdeq}) with $X_0 = 0\in\mathbb{R}^d$ and a constant drift and diffusion equal to $\mu_t = 0, \sigma_t = 1 \in\mathbb{R}^d$. The target function is a Gaussian function centered around a target point $z\in\mathbb{R}^d$:
$$h((X)_{0\leq t\leq 1}) = \exp\left(-\tfrac12 (X_1 - z)^\top (X_1 - z)\right).$$
\end{exmp}
For importance sampling we will consider two different classes of proposals $P^u$, corresponding with two different parameterizations of $u$ that are of the linear form $u=Ag(t, x)$ as detailed in Section~\ref{sect:control}. 

The first case that we consider is $g = 1\in\R$. The class of proposals that corresponds to $g = 1$ is in a sense the same class as we used in Example~\ref{exmp:counter}. It is a degenerate case of a diffusion process: since the control $u(t, x) = A\in\R^d$ is constant, all states, except the end state $X_{1}$, of the entire path $(X_t)_{0\leq t\leq 1}$ can be ignored, because $h$ is only a function of $X_1$. 

The second class is constructed with $g = (1, x) \in\R^{1 + d}$. The corresponding $u(t, x) = Ag$ with $A\in\R^{(d + 1)\times d}$ are linear feedback controls, making the intermediate states $X_t$ with $t<1$ relevant to the distribution of $X_1$. This more complex parametrization will give us more control over the process, and hence more flexibility in finding a good proposal. So although we need to learn a parameter $A$ with a higher dimension, we expect a higher ESS. 
\\

We use Example~\ref{exmp:bench} with $d = 3$, and $z = 2\cdot\mathds{1}\in\mathbb{R}^3$ (where $\mathds{1}\in\R^3$ is the vector with all ones) in order to compare four types of AMIS algorithms. 
All four methods will be implemented with the Path Integral Adaptation that is described in Appendix~\ref{appendix:algo}, using a parametrization based on $g = 1$. The difference between the four methods comes from the following different re-weighting schemes that they use:
\begin{itemize}
\item Balance re-weighting, with $N_k = 1$ sample per iteration. 
\item Optimized discarding time, i.e.~flat re-weighting with $t_k$ that maximizes $\widehat{\operatorname{ESS}}$ and $N_k = 1$ sample per iteration.
\item Flat re-weighting with discarding time $t_k = \lceil k/2\rceil$ and $N_k = 1$ sample per iteration. 
\item An iterative non-mixing scheme, with constant batch sizes $N_k$, where only the samples of the last iteration are used, i.e.~with $w_K\propto 1$ and $w_k = 0$ for $k<K$.
\end{itemize}
\begin{figure}
\begin{center}
\includegraphics[width=10cm]{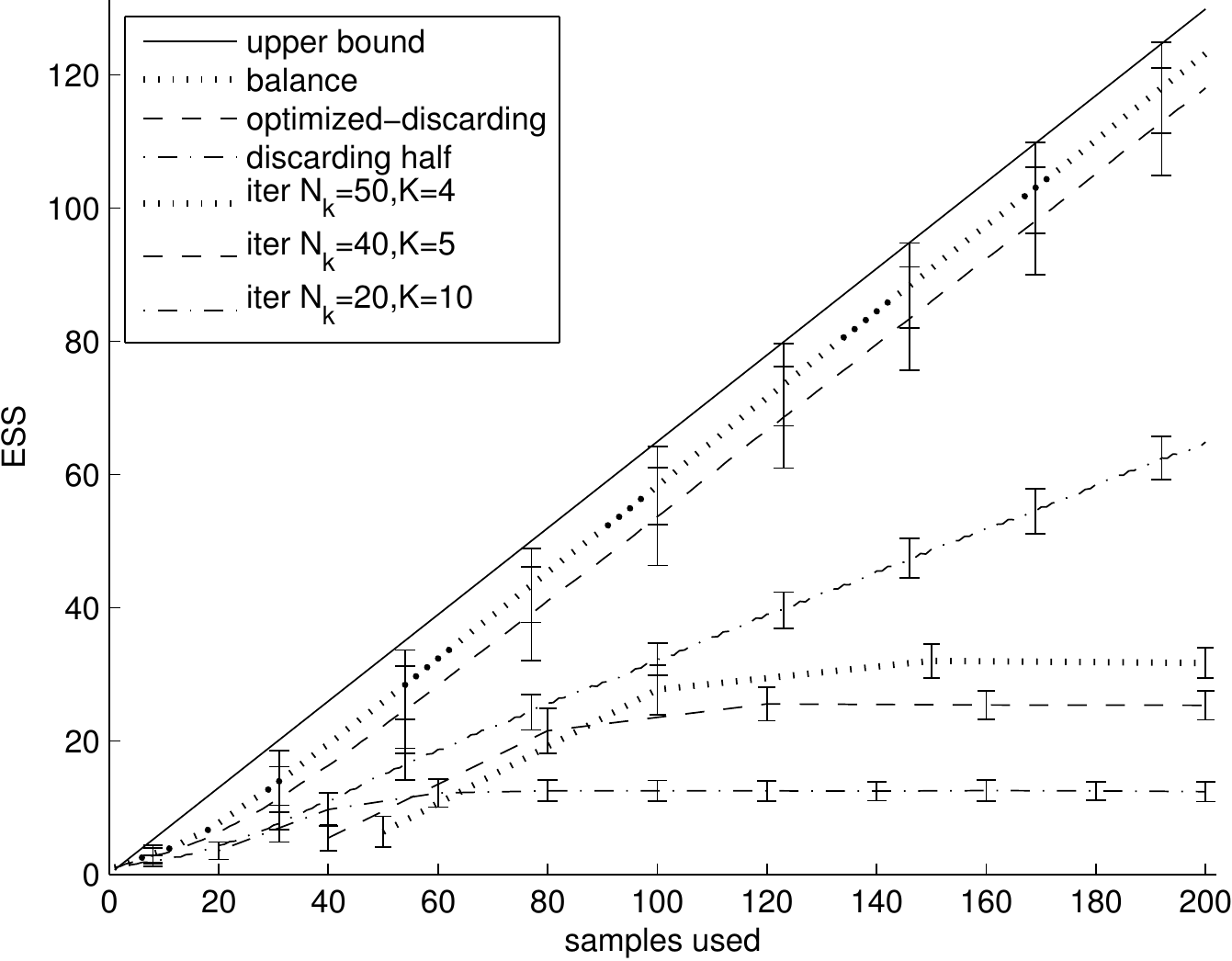}
\end{center}
\caption{
Average $\widehat{\operatorname{ESS}}$ of AMIS for various re-weighting schemes with adaptation based on $g = 1$, taken over 100 independent runs. On the $x$-axis are the number of samples that are used so far in the AMIS estimate, i.e., $\sum_{i = 1}^k N_i$ for $k = 1, \ldots, K$. 
}
\label{fig:ess}
\end{figure}

We report how the $\widehat{\operatorname{ESS}}$, averaged over $100$ runs, increased with the number of samples that was used by each of the AMIS algorithms, see Figure~\ref{fig:ess}. The upper bound is the optimal achievable ESS within the class of proposals corresponding to $g  = 1$. The slope of this upper bound is $\max\{\operatorname{ESS}^{P^u}\colon u = A,\ A\in\R^d\} = (3/4)^d$. The two methods that perform the best are balance-AMIS and optimized-discarding-AMIS. These two methods make optimal use of new samples when the underlying parameter $A$ has converged to it's optimal value, as can be seen by the slope that matches the upper bound. When discarding half of the samples, the slope in average $\widehat{\operatorname{ESS}}$ is halved as well. The iterative non-mixing schemes perform the worst: the average $\widehat{\operatorname{ESS}}$ never uses more than the $N_k$ samples of one batch in the estimate, and therefore $\widehat{\operatorname{ESS}}$ does not increase with $k$, or equivalently with the total number of samples used. 

A fifth method that we also tested uses flat re-weighting without discarding. Although this method seemed to perform reasonably on average, this result is not significant because of very large error bars. For these reasons this method is not included in Figure~\ref{fig:ess}. 
\\

\begin{figure}
\begin{center}
\includegraphics[width=10cm]{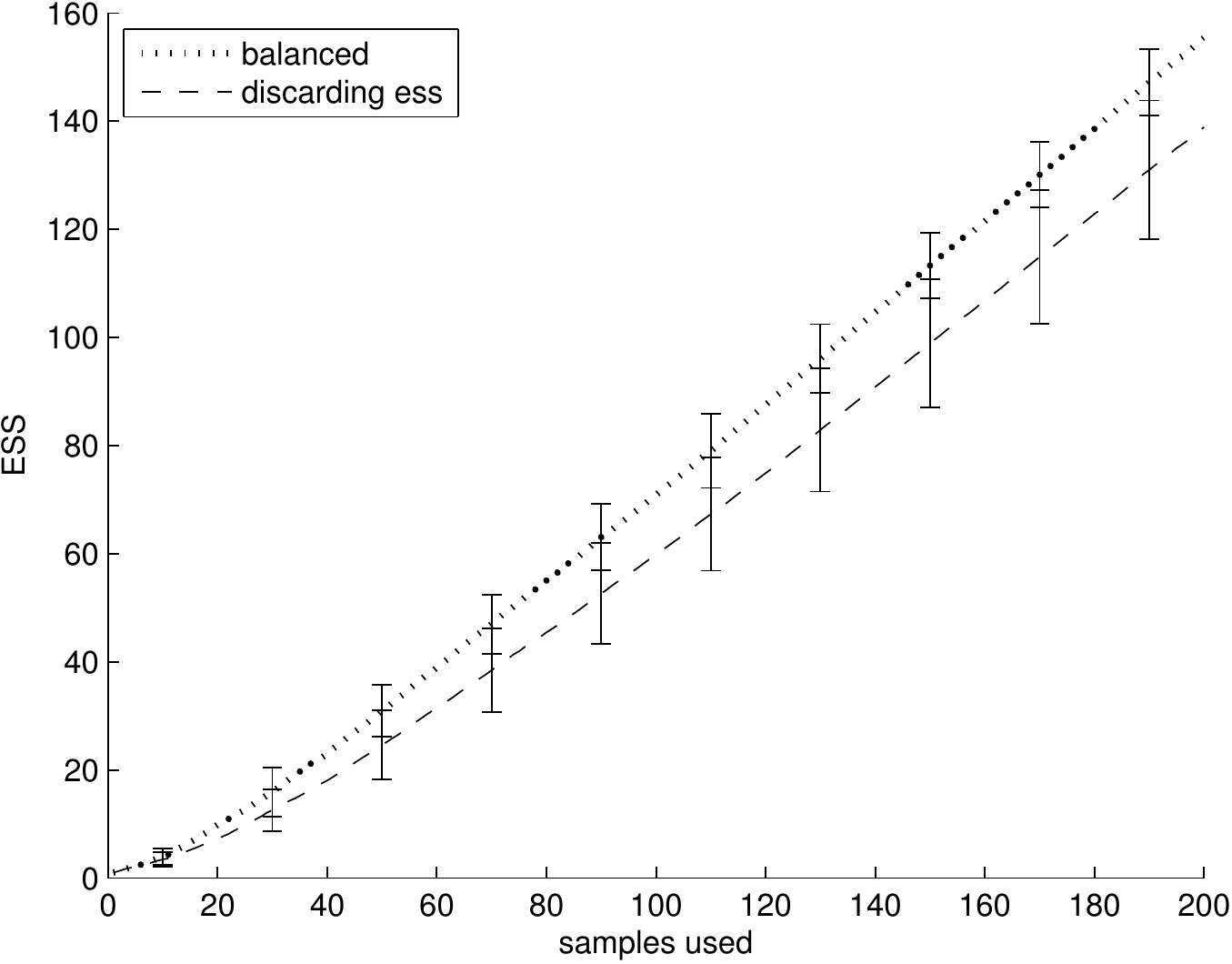}
\end{center}
\caption{
Average $\widehat{\operatorname{ESS}}$ of balance-AMIS and optimized-discarding-AMIS with adaptation based on $g = (x, 1)$, taken over 100 independent runs. On the $x$-axis are the number of samples that are used so far in the AMIS estimate, i.e., $\sum_{i = 1}^k N_i$ for $k = 1, \ldots, K$. 
}
\label{fig:essfeed}
\end{figure}
\begin{table}
\caption{Computation time in seconds for $100$ runs of AMIS based on $g = (1, x)$ for a fixed number $N = \sum_{k = 1}^K N_k = 200$ of samples, but with different numbers of iterations $K$.  \label{tab:time}}
\begin{center}
\begin{tabular}{r|cccccc}
$K$ 					& 10 	& 25 	& 50 	& 100 	& 200 	   \\\hline
balance 				& 46	& 104 	& 200 	& 392 	& 780   \\
optimized-discarding	& 5.9 	& 6.2	& 6.4 	& 6.9 	& 7.9   
\end{tabular}
\end{center}
\end{table}

We make the same comparison between optimized-discarding-AMIS with balance-AMIS while using the more complex parametrization with $g = (1, x)$. The results are reported in Figure~\ref{fig:essfeed}. Compared to the experiment with $g = 1$ we notice that the average $\widehat{\operatorname{ESS}}$ is higher, and that the two methods perform similar; balance-AMIS is slightly better. However, balance-AMIS also required a lot more computational resources to produce these results, see Table~\ref{tab:time}. From the table it appears that computation time for balance-AMIS is roughly proportional to the number of iterations. This is exactly what one would expect based on the complexity $\mathcal{O}(K^2M) = \mathcal{O}(KN)$ of the adaptation step, when the total number of samples $N=200$ is fixed. This complexity can be avoided in the adaptation step for optimized-discarding, because with that re-weighting scheme the weights are not changed in future iterations. We conclude that for a given number of samples optimized-discarding-AMIS performs almost as well as balance-AMIS, but at a fraction of the computational cost.

\appendix

\section{Proofs and definitions}\label{appendix:proof}

\begin{deff}\label{deff:mds}
Let $\{X_n\}_{n> 0}$ be a sequence of random variables that is adapted to the filtration $\{\mathcal{F}_n\}_{n> 0}$. Then $X$ is called a Martingale Difference Sequence w.r.t.~$\mathcal{F}$ when for all $n> 0$
\begin{enumerate}
\item $\E[X_n] < \infty$,
\item $\E[X_{n + 1}\mid\mathcal{F}_{n}] = 0$.
\end{enumerate}
\end{deff}

\begin{thrm}[Generalized Strong Law of Large Numbers]\label{thrm:slln} 
Let $\{Y_n\}_{n> 0}$ be a Martingale Difference Sequence relative to $\{\mathcal{F}_n\}_{n> 0}$. If $\{Y_n\}_{n> 0}$ is uniformly bounded in the $L^r$-norm for some $r > 1$, then 
\begin{align*}
\frac1N \sum_{n = 1}^N Y_n \to 0\text{ almost surely, as }N \to \infty. 
\end{align*}
\end{thrm}
\begin{proof}
This is proven in \cite{hansen} for Mixingale Sequences, which are more general then Martingale Difference Sequences. 
\end{proof}
\begin{rmrk}
The theorem above does not generalize to the case $r=1$. However, for $r =1$ there is a weak law (convergence in probability) when the set $\{Y_n\}_{n> 0}$ is uniformly integrable, see \cite{andrews}. 
\end{rmrk}
\begin{lemm}\label{lemm:klbound}
Let $\mathcal{U}$ be a set of adapted processes. Suppose that $\mathcal{U}$ is uniformly bounded in the $L^\infty$ norm, i.e. there is a constant $C$ such that for all $u\in\mathcal{U}$, 
$$\|u\|_\infty = \inf\left\{A\geq 0 : |u_t| < A \text{ for all } t, P^u\text{-almost surely}\right\} < C.$$
Then $dQ/dP^u$ is bounded uniformly in the $L^r$-norm for all $r\geq 1$. 
\end{lemm}
\begin{proof}
Let $r\geq 1$ and $u\in\mathcal{U}$. Choose $a$ such that $a>r$. 
\begin{align*}
\|dQ/dP^u\|_r 
	& = \left\|\exp\left(-\tint_0^Tu_t^{\top}dW_t - \tfrac12\tint_0^Tu_t^{\top}u_tdt\right)\right\|_{r} \\
	& = \left\|\exp\left(-\tint_0^Tu_t^{\top}dW_t - \tfrac{a}{2}\tint_0^Tu_t^{\top}u_tdt\right) \exp\left( - \tfrac{1 - a}{2}\tint_0^Tu_t^{\top}u_tdt\right)\right\|_{r}.
\end{align*}
Let $b$ be such that $\frac1r = \frac1a + \frac1b$. Then $b>1$, and by H\"olders Inequality and the bound on $u$ we obtain respectively
\begin{align*}
\|dQ/dP^u\|_r 
	&\leq\left\|\exp\left(-\tint_0^Tu_t^{\top}dW_t - \tfrac{a}{2}\tint_0^Tu_t^{\top}u_tdt\right)\right\|_a\left\|\exp\left(\tfrac{a - 1}{2}\tint_0^Tu_t^{\top}u_tdt\right)\right\|_{b} \\
	&<\left\|\mathcal{E}\left(-a\tint_0^Tu_t^{\top}dW_t\right)\right\|_1^{1/a} \exp\left((a - 1)C^2T/2\right). 
\end{align*}
The Dol\'ean exponential $\mathcal{E}\left(-a\int_0^Tu_t^{\top}dW_t\right)$ is a local martingale that is positive. Hence, it is a super martingale, so that
\begin{align*}
\|dQ/dP^u\|_r 
	&< \exp\left((a - 1)C^2T/2\right). \qedhere
\end{align*}
\end{proof}
\begin{proof}[Proof of Theorem~\ref{thrm:consistent_u}]
Note that Assumption 1 of the theorem implies Novikov's Condition, so that $P^u\sim Q$ for all $u\in\mathcal{U}$. So in particular we have $P^u\ll Q$ and and therefore the condition of Eq.~(\ref{eq:p<<qh}) holds. 

Now we will show that condition Eq.~(\ref{eq:bound}) also holds, so that consistency follows from Theorem~\ref{thrm:consistent}. Let $r>1$, and choose $a, b>1$ such that $\frac1r = \frac1a +\frac1b$. Then, using H\"olders inequality, we get
\begin{align*}
\left\|h\tfrac{dQ}{dP}\right\|_r
	& \leq \|h\|_{a} \left\|\tfrac{dQ}{dP}\right\|_{b}.
\end{align*}
Now, $\|h\|_{a}$ is bounded by Assumption~2 of the theorem, and $\left\|\tfrac{dQ}{dP}\right\|_{b}$ is bounded by Lemma~\ref{lemm:klbound}. 
\end{proof}

\begin{proof}[Proof of Theorem~\ref{thrm:control}]
Because $u^{\star}$ is an optimal control we have $P^{\star} = P^{u^{\star}}$. Furthermore, the equation 
$$dB^{\star}_t + u^{\star}_tdt = dB_t + u_tdt,$$ 
defines a $P^{\star}$-Brownian motion $B^{\star}_t$. If we multiply this equation with $g$, integrate, and take the expected value w.r.t.~$P^{\star}$, we obtain
\begin{align*}
\E_{P^{\star}}\left[\int_{0}^{T} g_t {u^{\star}_t}^\top dt\right] = \E_{P^{\star}}\left[\int_{0}^{T} g_t (u_t dt + dB_t)^\top\right].
\end{align*}
On the left hand side the $\int g_tdB^{\star}$ term vanished because it is a Martingale w.r.t.~$P^{\star}$ because $g$ is square integrable. 
The variable $h(X)dQ/dP^{\star}$, where $X\sim P^{\star}$, has zero variance, so it is safe to multiply the equation above with it
\begin{align*}
\E_{P^{\star}}\left[h(X)\frac{dQ}{dP^{\star}}\int_{0}^{T} g_t {u^{\star}_t}^\top dt\right] = \E_{P^{\star}}\left[h(X)\frac{dQ}{dP^{\star}}\int_{0}^{T} g_t (u_t dt + dB_t)^\top\right].
\end{align*}
Changing the measure from $P^{\star}$ to $P^u$ we obtain
\begin{gather*}
\E_{P^u}\left[h(X)\frac{dQ}{dP^{u}}\int_{0}^{T} g_t {u^{\star}_t}^\top dt\right] = \E_{P^u}\left[h(X)\frac{dQ}{dP^{u}}\int_{0}^{T} g_t (u_t dt + dB_t)^\top\right]. \qedhere
\end{gather*}
\end{proof}

\section{Path integral AMIS}\label{appendix:algo}

In this section we give a detailed description of the \index{path integral control algorithm}\textit{path integral control algorithm}. This algorithm computes MC estimates of expected values over a diffusion process. Therefore the algorithm can be used to solve a sampling problem, as described in Section~\ref{sect:wiener}.
In order to compute the MC estimates efficiently we will use adaptive multiple importance sampling (AMIS). Therefore the overall structure of the path integral control algorithm is very similar to Algorithm~\ref{algo:gamis} in Section~\ref{sect:gen-amis}. 

\begin{algorithm}
\caption{Path Integral Control Algorithm with AMIS}
\label{algo:pic}
\begin{description}
\item[Iterate.]
For $k = 1, \ldots, K $ do
	\begin{description}
	\item[Path Integral Adaptation.] \index{Path Integral Adaptation}
		Construct a feedback control $u_k(t, x)$, possibly depending on $X_n^l$ and $w_n^l$ with $1\leq l < k$ and $1\leq n\leq N_l$.
	\item[Generation.] 
		For $n = 1, \ldots, N_k$ draw sample paths $X_n^k\sim P_{k}$, where $P_k = P^{u_k}$,\\
		and compute the cost $S_n^k$ of the path.
	\item[Re-weighting.] \index{re-weighting}
		For $n = 1, \ldots, N_k$, construct $w_n^k$.\\
		For $l = 1, \ldots, k- 1$ and $n = 1, \ldots, N_l$, update $w_n^l$.
	\item[Intermediate Output.] Return $\hat{J}^{\star}_k = -\log\frac1{\sum_{l= 1}^kN_l}\sum_{l = 1}^k\sum_{n = 1}^{N_l} e^{-S_n^l}w_n^l$, \\
	and return $u_k(t, x)$. 
	\end{description}
\item[Output.] Return the optimal control estimate  $u_{k+1}(t, x)$ that would be computed by the Adaptation at iteration $K + 1$. 
\end{description}
\end{algorithm}

The path integral control algorithm runs $k = 1, \ldots, K$ iterations, and at each iteration four steps are executed. We shall give a description of each of these four steps. 

The first step is the \textbf{Path Integral Adaptation}, that implements the control computations for importance sampling. This step contains the core of path integral control. It is explained in more detail below. 

The second step is \textbf{Generation}. Here $N_k$ samples paths are drawn from the process that is given by a SDE of the form
\begin{align}
dX^u_t &= b\left(t, X_t^u\right)dt + \sigma\left(t, X_t^u\right) \left[\left(u(t, X_t^u\right)dt + dW_t\right], \label{eq:dyn} 
\end{align}
with importance control $u(t, x) = u_k(t, x)$. Note that this is a special case of Eq~\eqref{eq:sdep}. 
The cost of each path is calculated using 
\begin{align}
S^u_{t} &= Q\left(X^u_{t_1}\right) + \int_{t}^{t_1}\! R\left(\tau, X_\tau^u\right) + \frac12u\left(\tau, X_\tau^u\right)^\top\!\! u\left(\tau, X_\tau^u\right) d\tau + \int_{t}^{t_1}\! u\left(\tau, X_\tau^u\right)^\top\!\! dW_\tau.\label{eq:s}		
\end{align}
Note that the cost is the negative log of the importance weight as given in Eq.~\eqref{eq:rn}. 
Approximate samples from a diffusion process can for example be obtained with the Euler-Maruyama method, or similar higher order schemes, see \cite{kloeden}. 

The third step is the computation of the \textbf{Re-weighting}. This step is discussed in detail in Sections~\ref{sect:flat} and \ref{sect:discard}. 

The fourth and last step is the \textbf{Intermediate Output}, where we return two intermediate results: the estimate $\hat J$ of $-\log \mathbb{E}[e^{-S}]$, and the feedback control that was constructed in the first step.

\begin{algorithm}
\caption{Path Integral Adaptation (at iteration $k$)}
\label{algo:pia}
	\begin{itemize}	
	\item If $k = 1$ (initialization), set:\\
	$A_k = 0\in\mathbb{R}^{m\times l}$, \\
	$F_k = 0\in\mathbb{R}^{m\times l}$, \\
	$G_k = 0\in\mathbb{R}^{l\times l}$. \\
	\item Else, if $k > 1$,\\
	For $n = 1, \ldots, N_{k - 1}$, set
	\begin{align*}
	g_t^{n, k - 1} 
		&= g\left(t, \left(X_n^{k - 1}\right)_t\right),\\
	u_t^{n, k - 1} 
		&= A_lg_t^{n, k - 1},\\
	R_t^{n, k - 1}
		&= R\left(t, \left(X_n^{k -1 }\right)_t\right)\\
	S_n^{k - 1} 
		&= Q\left(X_{t_1}\right) + \int_{t_0}^{t_1} {u_t^{n, k - 1}}^\top \left(dW_n^{k - 1}\right)_t + \frac12 {u_t^{n, k - 1}}^\top u_t^{n, k - 1} dt + R_t^{n, k - 1}dt.
	\end{align*}
	And set
	\begin{align*}
	G_k 
		&= \sum_{l = 1}^{k - 1}\sum_{n = 1}^{N_l} h\left(X_n^{l}\right) e^{-S_n^l} w_n^{l} \int_{t_0}^{t_1} g_t^{n, l}{g_t^{n, l}}^\top dt,\\
	F_k 
		&= \sum_{l = 1}^{k - 1}\sum_{n = 1}^{N_l} h\left(X_n^{l}\right) e^{-S_n^l} w_n^{l} \int_{t_0}^{t_1} \left(u_t^{n, l}dt  + \left(dW_n^{l}\right)_t\right){g_t^{n, l}}^\top dt,\\
	A_k 
		&= F_kG_k^{-1}.
	\end{align*}
	Let $u_k(t, x) = A_kg(t, x)$.\\
	Let $P_k$ be the measure induced by Eq.~(\ref{eq:dyn}) with $u = u_k$.  
	\end{itemize}	
\end{algorithm}

Next, we give a more detailed description of the \textbf{Path Integral Adaptation} step (see Algorithm~\ref{algo:pia}) and it's interaction with Algorithm~\ref{algo:pic}. The path integral adaptation is the core of Algorithm~\ref{algo:pic}: here the samples are combined adaptively to construct new controls $u_k$ that induce new proposal distributions $P_k$. The adaptation  requires a user specified parametrization of the control $u(t, x) = Ag(t, x)$, for a given function $g(t, x)$. The main task of Algorithm~\ref{algo:pia} is to compute a AMIS MC estimate $A_k$ that is subsequently used to define the next control. This procedure is repeated sequentially as Algorithm~\ref{algo:pic} runs trough the $K$ iterations. 
At the $k$-th iteration $N_k$ samples are drawn from the proposal distribution $P_k = P^{u_k}$, where $u_k(t, x) = A_kg(t, x)$. The parameters $A_k$ are computed with samples from previous iterations, and are an AMIS estimate of the optimal $A^\star$. The computation of the estimate $A_k$ is based on Eq.~(\ref{eq:a}), that is equivalent to the three equations below. 
\begin{align*}
A^{\star} &= F^{\star}{G^{\star}}^{-1}, \\
F^{\star} &= \E_{P^u}\left[h(X)\frac{dQ}{dP^u} \int_{t_0}^{t_1} (u_tdt + dW_t)g_t^\top\right],\\
G^{\star} &= \E_{P^u}\left[h(X)\frac{dQ}{dP^u} \int_{t_0}^{t_1} g_tg_t^\top dt\right]. 
\end{align*}
At each iteration $k$ the terms $F^{\star}$ and $G^{\star}$ will be estimated by $F_k$ and $G_k$ respectively. 

Note that the path weights $e^{-S_n^l}$ are also required in the \textbf{Intermediate Output} step of Algorithm~\ref{algo:pic}. The time integrals above could, for example, be estimated approximately with the Euler-Maruyama method. We remark that the algorithm above is of order $\mathcal{O}(MK^2)$ when $N_k = M$ for all $k$, because it is given for a generic re-weighting scheme. When flat- or discarding-re-weighting is used, $G_k$ and $F_k$ might be computed incrementally because the re-weighting does not change once set, dropping the first sum, so that the algorithm becomes $\mathcal{O}(MK)$.

\bibliographystyle{alpha}
\bibliography{diffusion_amis_new}

\end{document}